\documentclass[reqno,twoside,final]{amsart}
\usepackage{amsmath,amstext,amsthm,amsfonts,amssymb,amscd}

\newcommand{\etype}[1]{\renewcommand{\labelenumi}{(#1{enumi})}}
\def\eroman{\etype{\roman}}

\newcommand{\Image}{{\operatorname{Im}\ }}
\newcommand{\Char}{{\operatorname{Char}\ }}
\newcommand{\Cent}{{\operatorname{Cent}}}

\newcommand{\diag}{{\operatorname{Diag}}}

\theoremstyle{definition}

\newtheorem{defn}{Definition}\newtheorem{cor}{Corollary}
\newtheorem{rem}{Remark}
\newtheorem*{rem*}{Remark}
\newtheorem*{acknow*}{Acknowledgements}
\newtheorem*{examples*}{Examples}
\newtheorem{examples}{Example}
\theoremstyle{plain}

\newtheorem{prm}{Problem}
\newtheorem{lemma}{Lemma}
\newtheorem{prop}{Proposition}
\newtheorem{theorem}{Theorem}
\newtheorem*{theorem*}{Theorem}
\newtheorem{conjecture}{Conjecture}

\newenvironment{proof-sketch}{\noindent{\bf Sketch of Proof}\hspace*{1em}}{\qed\bigskip}
\newenvironment{proof-idea}{\noindent{\bf Proof Idea}\hspace*{1em}}{\qed\bigskip}
\newenvironment{proof-of-lemma}[1]{\noindent{\bf Proof of Lemma #1}\hspace*{1em}}{\qed\bigskip}
\newenvironment{proof-of-prop}[1]{\noindent{\bf Proof of Proposition #1}\hspace*{1em}}{\qed\bigskip}
\newenvironment{proof-of-thm}[1]{\noindent{\bf Proof of Theorem #1.}\hspace*{1em}}{\qed\bigskip}
\newenvironment{proof-attempt}{\noindent{\bf Proof Attempt}\hspace*{1em}}{\qed\bigskip}

\def\a{\alpha}
\def\Cent{\operatorname{Cent}}
\def\UD{\operatorname{UD}}
\def\sl{\operatorname{sl}}
\def\bfi{\bold i}
\def\bfj{\bold j}

\begin{document}

\title[Power-central polynomials]{Power-central polynomials on matrices}
\author{Alexey Kanel-Belov, Sergey Malev, Louis Rowen}

\address{Department of mathematics, Bar Ilan University,
Ramat Gan, Israel} \email{beloval@math.biu.ac.il}
\email {malevs@math.biu.ac.il}
\email{rowen@math.biu.ac.il}
\thanks{This research was supported by the
Israel Science Foundation  (grant no. 1207/12).}
\thanks{ The second named author was supported by an Israeli Ministry of Immigrant Absorbtion
scholarship.}

\begin{abstract} Any  multilinear non-central polynomial  $p$ (in several noncommuting
variables) takes on values of degree $n$ in the matrix algebra
$M_n(F)$ over an infinite field $F$. The polynomial  $p$  is
called {\it $\nu$-central} for $M_n(F)$ if $p^\nu$ takes on only
scalar values, with $k$ minimal such. Multilinear $\nu$-central
polynomials do not exist for any $\nu$ with $n>3$, thereby
answering a question of Drensky. Saltman proved that an arbitrary
polynomial
  $p$   cannot be
$\nu$-central for $M_n(F)$ for $n$ odd unless $n$ is prime; we
show for $n$ even, that $\nu$ must be 2.
\end{abstract}


\maketitle

\section{Introduction}

For any polynomial $p\in K\langle x_1,\dots,x_m\rangle$, the
$image$ of $p$ (in $R$), denoted  $\Image p$, is defined as
$$\{r\in R:
\ \text{there exist}\ a_1,\dots,a_m\in R\ \text{such that}\
p(a_1,\dots,a_m)=r\}.$$

\begin{rem}\label{cong}
$\Image p$ is invariant under conjugation, since $$a
p(x_1,\dots,x_m)a^{-1}=p(a x_1a^{-1},a x_2a^{-1},\dots,a_m
a^{-1})\in\Image p,$$ for any nonsingular $a\in M_n(K)$.
\end{rem}

This note is an outgrowth of \cite{BMR1} and \cite{BMR2}, in which
we considered the question, reputedly raised by Kaplansky, of
$\Image p$   on  $n \times n$ matrices. (We take $n$ to be given
throughout this note, as is our base field $K$.)

Specifically, the following conjecture is attributed to Kaplansky:
\begin{conjecture}\label{Polynomial image}
If $p$ is a multilinear polynomial evaluated on the matrix ring
$M_n(K)$, then $\Image p$ is either $\{0\}$, $K$ (viewed as $K$
the set of scalar matrices), $\sl_n(K)$, or $M_n(K)$.
\end{conjecture}

Here is one celebrated example.

\begin{defn} A polynomial $p$ is \textbf{central} (with respect to $M_n(K)$)  if $p$
takes on only scalar values, but does not vanish identically.
\end{defn}

The existence of multilinear central polynomials for $n\times n$
matrices was proved by Formanek~\cite{F1} and
Razmyslov~\cite{Ra1}.
 Furthermore, it has long been known that $[x_1,x_2]^2$ is
central for $2 \times 2$ matrices, and  one of the possibilities
arising for $3 \times 3$ matrices is a polynomial whose cube is
central, thereby motivating the following definition:

\begin{defn} A polynomial $p$ is \textbf{$\nu$-central} if $p^\nu$
is central, for $\nu\ge 1$ minimal such. The polynomial $p$ is \textbf{power-central} if $p$
is $\nu$-central, for some $\nu >1.$ 
\end{defn}

Our objective here is to examine the existence of $\nu$-central
polynomials. For $n=\nu$ prime, as explained in
\cite[Theorem~3.2.20]{Row1}, this is equivalent to Amitsur's
generic division algebra being cyclic, one of  the major open
questions in the theory of division algebras, and we have nothing
more to say about this case, but we can solve the other cases.

\begin{rem}\label{genericten00} A homogeneous 3-central polynomial for $n=3$ was constructed in
\cite[Theorem~3.2.21]{Row1}, and a homogeneous 2-central
polynomial for $n=4$ was constructed in
\cite[Proposition~3.2.24]{Row1}.
\end{rem}

 Using the structure theory of division algebras,
Saltman~\cite{S} proved that in characteristic 0, $\nu$-central
polynomials do not exist for odd $\nu>1$ unless $n$ is prime. This
led Drensky to ask what happens for $\nu=2$, noting the result of
Remark~\ref{genericten00}. We consider both  the homogeneous case
and the more restrictive multilinear case.

 Our main result is:

 \bigskip
{\bf Theorem \ref{no-mpc}.}
 Assume $n\geq 4.$ For $ \Char(K)$ arbitrary,  there are no multilinear power central polynomials.
 Any multilinear polynomial is either PI, or central, or its image in $M_n(K)$  is at least $(n^2-n+2)$-dimensional.
\bigskip

This implies easily that multilinear $\nu$-central
 polynomials do not exist unless $\nu=n$. One can conclude
 via structure theory that:

(Theorem~\ref{generictena})  4-central polynomials do not exist,
and (Theorem~\ref{thmC}) multilinear 2-central
 polynomials do not exist unless $n=2$.

 Then, by examining dimensions of images, we conclude in Theorem~\ref{no-mpc}, for arbitrary $\nu$, that
multilinear $\nu$-central
 polynomials do not exist whenever $n \ge 4.$

\section{Considerations arising from division algebras}

These questions (for $\nu$ a power of 2) have easy answers for
homogeneous polynomials. Much of the material in this section is
  standard (although we do not have a specific reference),
obtained from well-known facts about division algebras. We say an
element $d$ of a division algebra $D$ with center $F$ is
\textbf{$\nu$-central} if $d^\nu \in F$ but $d^\ell \notin F$ for any
$\ell$ dividing $\nu$, $\nu$ minimal such. We say a subspace $V$ is \textbf{$\nu$-central}
if every element is $\ell$-central for some $\ell$ dividing $\nu$.

 A  major tool is  Amitsur's Theorem \cite[Theorem
3.2.6, p.~176]{Row1}, that the algebra of generic $n\times n$
matrices (generated by matrices $Y_k = (\xi _{i,j}^{(k)})$ whose
entries $\{ \xi _{i,j}^{(k)}, 1 \le i,j \le n\}$ are commuting
indeterminates) is a non-commutative domain $\UD$ whose ring of
fractions with respect to the center is a division algebra which
we denote as $\widetilde{\UD}$ of dimension $n^2$ over its center
$F_1: = \Cent(\widetilde{\UD}))$.

We shall need a general fact about polynomial evaluations.
\begin{lemma}\label{noncom} For any polynomial  $p(x_1, \dots, x_m)$ which has an evaluation of degree $n$ on $M_n(F)$,
there is an index $i,\ 1\le i\le m,$ and matrices $a_1, a_2, \dots, a_m,a_i'$ such that the evaluations
$ p(a_1, \dots, a_{i-1},a_i,a_{i+1},\dots, a_m)$ and $ p(a_1, \dots, a_{i-1},a_i',a_{i+1},\dots, a_m)$ do not commute.
\end{lemma}
\begin{proof} We go back and forth to generic matrices and  $\widetilde{\UD}$. First of all, for all generic matrices $Y_1, \dots, Y_m, Y_1', \dots, Y_m'$,
and each $i$,
clearly $ p(Y_1, \dots,   Y_i,Y'_{i+1},\dots, Y'_m)$ has degree $n$ over $F_1$, and thus has distinct
eigenvalues, from which it follows at once that $ p(Y_1, \dots,  Y_m)$ and $ p(Y_1',\dots, Y'_m)$ do not commute (since
one could diagonalize $ p(Y_1, \dots,  Y_m)$  while $ p(Y_1',\dots, Y'_m)$   remains non-diagonal). 

But, for each $i$,  $F_1( p(Y_1, \dots,  ,Y_i,Y'_{i+1},\dots, Y'_m))$  has dimension $n$  over $F_1$  and thus is a maximal subfield of  $\widetilde{\UD}$. It follows that  $ p(Y_1, \dots, Y_i,Y'_{i+1},\dots, Y'_m) $ and $ p(Y_1, \dots, Y_{j-1},Y'_{j},\dots, Y'_m)$ commute iff  $$F_1( p(Y_1, \dots, Y_i,Y'_{i+1},\dots, Y'_m)) = F_1( p(Y_1, \dots, Y_{j-1},Y'_{j},\dots, Y'_m)).$$ In particular,  $F_1(  p(Y_1, \dots,  Y_m))  \ne F_1( p(Y'_1, \dots,  Y'_m) )$,
implying  $$F_1( p(Y_1, \dots, Y_i,Y'_{i+1},\dots, Y'_m))\ne F_1( p(Y_1, \dots, Y_{i-1},Y'_{i},\dots, Y'_m))$$ for some $i$, and thus  
$ p(Y_1, \dots, Y_i,Y'_{i+1},\dots, Y'_m) $ and $ p(Y_1, \dots, Y_{i-1},Y'_{i}\dots, Y'_m)$ do not commute. In other words,
$$[ p(Y_1, \dots, Y_i,Y'_{i+1},\dots, Y'_m) , p(Y_1, \dots,Y_{i-1},Y'_{i}\dots, Y'_m)] \ne 0,$$ implying there is a specialization 
$$Y_1 \mapsto a_1, \dots, Y_i \mapsto a_i  ,   Y_i' \mapsto a_i', Y'_{i+1} \mapsto a_{i+1}, \dots, Y'_m \mapsto a_m$$ 
yielding $[ p(a_1, \dots, a_{i-1},a_i,a_{i+1},\dots, a_m),p(a_1, \dots, a_{i-1},a_i',a_{i+1},\dots, a_m)] \ne 0.$ 
\end{proof}

We return to power-central polynomials.

\begin{rem}\label{genericten0}
The existence of a $\nu$-central polynomial is equivalent to $\UD$
containing an element whose $\nu$-power is central, with $\nu$ minimal
such. On the other hand, any such element can be specialized to an
arbitrary division algebra of dimension $n^2$ over its center.
Thus, to prove the non-existence of a $\nu$-central polynomial, it
suffices for suitable $\ell $ dividing $\nu$ to construct a division
algebra with center $F \supseteq K$ having the property that if
$d^\nu \in F$ then $d^\ell \in F.$
\end{rem}

\begin{examples}\label{genericten} Suppose $K$ has characteristic $\ne 2$, and $n = 2^{t-1} q$ where $q$ is odd, and
construct $D$ to be a tensor product of ``generic'' symbols
$(\lambda_u^{n_u} , \mu_u^{n_u} )$ as in \cite[Example~7.1.28 and
Theorem~7.1.29]{Row1}, where  $n_1 = n_2 = \dots = n_{t-1} = 2$
and $n_t = q.$ In other words, $D$ is the algebra of central
fractions of the skew polynomial ring $R :=
K(\rho)[\lambda_u\mu_u: 1 \le u \le t]$ where $\rho$ is a
primitive $q$ root of 1 and the indeterminates commute except for
$\lambda_u \mu_u = - \mu_u \lambda_u$ for $1 \le u \le t-1$ and
$\lambda_t \mu_t =\rho \mu_t \lambda_t$. We write a typical
element of $R$ as $\sum \a _{\bfi,\bfj} \lambda ^{\bfi} \mu
^{\bfj}$, where $\lambda ^{\bfi}$ denotes $\prod_{u=1}^t
\lambda_u^{i_u}.$ There is a natural grade given by the
lexicographic order on the exponents of the monomials, and it is
easy to see that if $d^\nu \in F$ then the leading term $\hat d^\nu
\in F.$

In particular, for $\nu =2,$ if $d^2 \in F$ then $\hat d$ must have
the form $\a \lambda ^{\bfi} \mu ^{\bfj}$ where $i_t = j_t = 0.$
On the other hand, we claim that if $ d ^2 \in F$ and $\hat d \in
F$, then $d \in F.$ Indeed, taking $d'$ to
 be the next leading term in $d$, we   have
 $$d^2 = (\hat d + d')^2 = \hat d^2 + 2\hat d d' + \cdots,$$
 implying $d' \in F$, and continuing, we   conclude $d\in F$,
as desired.

 It follows that if $d$ is $2$-central then $\hat d$ is $2$-central.

 Now we claim that $D$ does not have $4$-central elements. Indeed,
 if $d$ is 4-central   then $d^2$ is $2$-central, implying $\hat d^2$ is
 $2$-central,
 and thus $\hat d$ is $4$-central, implying $\hat d$ must have
the form $\a \lambda ^{\bfi} \mu ^{\bfj}$ where $i_t = j_t = 0;$
we conclude that $\hat
 d$ is is $2$-central, implying $\hat d^2 \in F,$ and thus $d^2 \in
 F$ by the claim.
\end{examples}

\begin{theorem}\label{generictena} There do not exist $4$-central
polynomials.
\end{theorem}
\begin{proof} Combine Remark~\ref{genericten0}
with Example ~\ref{genericten}.\end{proof}

This leaves us with 2-central polynomials.

\begin{prop}\label{generictena1} There   exist homogeneous 2-central
polynomials with respect to $M_n(F)$ if  $n = 2q$ or  $n = 4q$ for
$q$ odd.
\end{prop}
\begin{proof} $\widetilde{\UD}$ is a tensor product of a division algebra $D_1$ of degree 2 or  4
and a division algebra of degree $q$, and we observed earlier that
$D_1$ has a $2$-central element. \end{proof}

The situation for $8|n$ remains open, and is equivalent to another
important question in division algebras about the existence of
square-central elements. The following observation might be
relevant, although we do not use it further.

\begin{lemma} Suppose $R = M_2(D)$ with $D$ an $F$-central division algebra.
The 2-central elements of $R$ have the form
\begin{equation}\label{2pol0}\left(
\begin{matrix} a & b
\\ -b^{-1}(a^2 - \a) &  -b^{-1}a b \end{matrix}\right) : \a \in F, a, b \in
D.\end{equation}
\end{lemma}
\begin{proof} Consider any square-central matrix $A =\left(
\begin{matrix} a & b
\\ c &  d \end{matrix}\right)$.
Then $$A^2 =\left(
\begin{matrix} a^2 + bc & ab + bd
\\ ca + dc & cb + d^2 \end{matrix}\right),$$
so we must have $\a \in F$ such that:
\begin{enumerate}\eroman
\item $a^2 + bc = \a$; \item $ cb + d^2 = \a$; \item $ab + bd =
0;$ \item $ca + dc = 0.$
\end{enumerate}

(i) implies $c = -b^{-1}(a^2 - \a) $. Then  (ii) implies

$$d^2 = \a + b^{-1}(a^2 - \a) ) b = b^{-1}a^2 b  = (b^{-1}a
b)^2,$$ so $d = \pm b^{-1}a b.$ But only $d = - b^{-1}a b$ works.
Thus the matrix is in the form of \eqref{2pol0}, and this is
indeed square-central.
\end{proof}

Note that this enables one to construct ``generic'' 2-central
elements, by choosing $a$ and $b$ and $\a$ arbitrarily, but we do
not know if they occur as elements of $\widetilde{\UD}.$

\section{Multilinear $2$-central polynomials}

Having  settled the issue for homogeneous polynomials except for
$n = 8q$, we turn to multilinear polynomials, where the story ends
differently. Although we will get a more general result, it is
instructive to start with the division algebra approach.

\begin{lemma} Any division algebra $D$ with a 2-central subspace $V$ of dimension 2
contains   an $F$-central quaternion subalgebra. In particular, $n
: = \deg(D)$ cannot be odd. Also, 4 does not divide $n$ if $D$
also has exponent $n$.
\end{lemma}

\begin{proof} Take $v, v' \in V $. Then $v^2,
{v'}^2 \in F.$ But also, by assumption, $v+v' $ is also
square-central, so
$$v^2 +vv' + v'v +{v'}^2 = (v+v')^2 \in F,$$ implying $v'v = -vv' +
\alpha$ for some $\alpha \in F.$ But then $F + Fv + Fv' + Fvv'$ is
a central $F$-subalgebra of $D$ and has dimension at least 3, but
has elements of degree 2, so has dimension 4.

The last assertion follows easily from the theory of finite
dimensional division algebras. If $D$ has a quaternion division
algebra then 2 must divide $n$ and the exponent of $D$ is the
least common multiple of 2 and $\frac n2.$
\end{proof}

\begin{lemma}\label{onev} If $p(x_1, \dots, x_m)$ is a 2-central polynomial for $n\times n$
matrices,  linear in $x_1$, and there are non-commuting values
$p(a_1, \dots, a_m)$  and $p(a_1', \dots, a_m)$ for matrices $a_1,
a_1', a_2, \dots, a_m,$ then the generic division algebra
$\widetilde{\UD}$ of degree $n$ has a $F_1$-central quaternion
subalgebra. In particular, $n$ cannot be odd, and 4 does not
divide $n$.
\end{lemma}

\begin{proof} Let $$w = p(Y_1, \dots, Y_m), \quad w' = p(Y_1', Y_2 \dots,
Y_m),$$ where the $Y_i$ and $Y_1'$ are generic matrices. Then
$w^2, {w'}^2 \in F_1.$ But also, by definition, $w+w' =
p(Y_1+Y_1', \dots, Y_m)$ is also 2-central, so $$w^2 +ww' + w'w
+{w'}^2 = (w+w')^2 \in F_1,$$ implying $w'w = -ww' + \alpha$ for
some $\alpha \in F_1.$ But then $F_1 + F_1w + F_1w' + F_1ww'$ is a
central $F_1$-subalgebra of $\widetilde{\UD}$ and has dimension 4.

The last assertion follows   since the generic division algebra
$\widetilde{\UD}$ of degree $n$ has exponent~$n$, whereas if
$\widetilde{\UD}$ has a central quaternion division subalgebra,
then 2 must divide $n$ and the exponent of $\widetilde{\UD}$ is
the least common multiple of~2 and~$\frac n2.$
\end{proof}

\begin{theorem}\label{thmB} If $p(x_1, \dots, x_m)$ is a multilinear 2-central polynomial for $n\times n$ matrices, then the
generic division algebra of degree $n$ has a quaternion part. In
particular, $n$ cannot be odd, and 4 does not divide $n$.
\end{theorem}
\begin{proof}
In view of Lemma~\ref{noncom}, two consecutive terms of the chain
$$p(a_1, \dots, a_m), p(a'_1,a_2 \dots, a_m), \cdots, p(a'_1,
\dots, a'_m)$$   do not commute, so conclude with
Lemma~\ref{onev}.
\end{proof}

This conclusion is the opposite of Proposition~\ref{generictena1},
when $n=4q$ for $q$ odd.

For $n>2$, we also have an easy consequence of the theory of
division algebras.

\begin{lemma}\label{div}  If a polynomial $p$ is $\nu$-central for $M_n(K)$ for $\nu>1$, then $\nu$
cannot be relatively prime to $n$.
\end{lemma}
\begin{proof}  We can view $p$ as an element of the generic
division algebra $\widetilde{\UD}$ of degree~$n$, and we adjoin an
$\nu$-root of 1 to $K$ if necessary. Then $p$ generates a subfield of $\widetilde{\UD}$,
of dimension dividing $\nu$. Hence the dimension is 1;
i.e., $\nu=1$.
\end{proof}

\section{Multilinear polynomials evaluated on $n\times n$ matrices}

 To handle the remaining
2-central case, where $n =2q$ for $q$ odd, we need to dig deeper
into the computations. We   prove a result that also applies for
arbitrary~$m$, for $n>3.$
 We need the following well-known lemma about
Eulerian graphs, cf.~{\cite[Lemma 4]{BMR1}}:

\begin{lemma} \label{graph}
If  $a_i$ are matrix units, then $p(a_1,\dots,a_m)$ is either $0$
or a diagonal matrix, or~$\alpha e_{ij}$ for some $\alpha\in K$
and $i\neq j$.
\end{lemma}
 
\begin{theorem}\label{thmB1}
Let $p(x_1,\dots,x_m)$ be any multilinear polynomial evaluated on
$n\times n$ matrices over an infinite field. Assume that $p$ is
neither scalar nor PI. Then $\Image p$ contains a matrix with
eigenvalues $\{c,c\varepsilon,\dots,c\varepsilon^{n-1}\}$ for some
$0 \ne c \in K$.
\end{theorem}
\begin{proof}
Define $\chi$ to be the permutation  of the set of matrix units,
sending the index $i\mapsto i+1$ for $1\leq i\leq n-1$, and
$n\mapsto 1$. For example, $\chi(e_{12})=e_{23},\
\chi(e_{57})=e_{68}$. Since we substitute only matrix units into $p$, Lemma \ref{graph}
shows that the image is either diagonal or a matrix unit with some
coefficient. Consider the corresponding graph~$\Gamma$.  If the graph is an Eulerian cycle then this sum is $0$,
and if it is an Eulerian path from $k$ to $\ell$ then this sum equals
$\ell-k$. By Lemma \ref{graph} there   exist matrix
units $a_1,\dots,a_m$ with $p(a_1,\dots,a_m)=\a e_{12}$ for some $\a \in K$.
We may assume that $ a_1\cdots a_m=e_{12}.$ Writing $a_\ell = e_{i_\ell, j_\ell},$
we define $\iota (a_\ell) = i_\ell - j_\ell.$  Thus $i_1 = 1$ and $j_m = 2,$ and   
$\sum \iota (a_\ell) = 1.$  Then $ \chi ^k (a_\ell) = e_{i_\ell+k, j_\ell+k}$ (taken modulo $n),$
implying  $$\iota (\chi ^k (a_\ell) ) \equiv  ( i_\ell+k) - (j_\ell+k) =  i_\ell - j_\ell = \iota (a_\ell) \pmod n. $$

consider

\begin{equation}\label{mapping_sets}
f(a_1,\dots,a_m)
=p\left( \sum _{k_1=1}^m t_{k_1,1}\chi ^{k_1} ( a_1),\dots,
 \sum _{k_m=1}^m t_{k_m,m}\chi ^{k_m} ( a_m)\right),
\end{equation}
where the $t_{k,\ell}$ are commuting indeterminates.
Opening the brackets, we   have $n^m$ terms, each of the form 
$$a' = \chi^{k_1}(a_{\pi(1)})\cdots \chi^{k_m}(a_{\pi(m)})$$
which, if nonzero, must have  $$\iota(a') \equiv  \sum _{\ell = 1}^m \iota (\chi^{k_\ell} a_{\pi(\ell)})\equiv   \sum _{\ell = 1}^m \iota (a_{\pi(\ell)}) \equiv  1\pmod n, $$
implying $a'$ is a matrix of the form $ce_{i,i+1}$ or $ce_{n,1}$. 
 Hence $\Image f\subseteq\Image p$ has the form
$$ a = \left(
\begin{matrix} 0 & * & 0 & \dots & 0
\\  0 & 0 & * &   \dots & 0   \\   \vdots &  \vdots &  \ddots &  \ddots & \vdots  \\
 0 & 0 & 0 & \ddots  & *  \\ * & 0 &  \dots & 0 & 0  \end{matrix}\right)
 .$$
Each of the starred entries of $a$ is a polynomial with respect to
$t_{k,i}$ and each of them takes nonzero values because
$e_{k,k+1}$ belongs to the image of $f$ for any $1\leq k\leq n-1$
and  also $e_{n,1}\in\Image(f)$. Therefore for generic $t_{k,\ell}$
each of the starred entries is nonzero, so the minimal polynomial
of $a$ is $\lambda ^ n - \alpha$ for some $\a$, implying $a$ has
eigenvalues $\{c,c\varepsilon,\dots,c\varepsilon^{n-1}\}$ where
$c$ is the $n$-th root of the determinant $\alpha$.
\end{proof}

\begin{rem}\label{dim-f0} The variety of $n\times n$ matrices with a given set of $n$ distinct
eigenvalues has dimension $n^2-n$. 
  \end{rem}

\begin{rem}\label{dim-f1} 
 Assume for some matrix units $a_i$ that $p(a_1,\dots,a_m)$ is diagonal.
 Then    $f$ as constructed in \eqref{mapping_sets} in the proof of Theorem \ref{thmB1}
 is diagonal.
If the dimension of the image of $\Image  f$ is $\delta$, then each value of $f$ has some set 
 of eigenvalues of $p$ and therefore by  Remark~\ref{dim-f0},  $\Image p$ has dimension at least $n^2-n+\delta.$

As a special case, if $p$ is power-central then   $\Image p$ has dimension  $n^2-n+1.$
\end{rem}

\begin{theorem}\label{thmC}
If a multilinear polynomial $p$ is $\nu$-central on $M_n(K)$, where
$\Char(K) $ does not divide $ n,$ then  then $\nu=n$.
\end{theorem}
\begin{proof} Passing to the algebraic closure, we apply
Theorem~\ref{thmB1} to Lemma~\ref{graph}.
\end{proof}
\begin{cor}
If a multilinear polynomial $p$ evaluated on $n\times n$ matrices
is $\nu$-central then $\nu=n$.
\end{cor}

\begin{proof}
The matrices in the image can only have $\nu$ distinct eigenvalues,
contradicting the theorem unless $\nu = n.$
 \end{proof}

Let us introduce  the tool of
``harmonic bases'' of the space of diagonal matrices.

\begin{rem}\label{remharmonic} Assume that $K $ has the form
$F[\varepsilon],$ where $\varepsilon$ is a primitive $n$-th root of~1. Let $\overline{\phantom w}$ denote the automorphism of $K$
sending $\varepsilon\mapsto \varepsilon^{-1}.$
 Take the  base of the diagonal matrices  $e_k=\diag\{1,\varepsilon^k,\varepsilon^{2k},\dots,
 \varepsilon^{(n-1)k}\},$  $0\leq k\leq n-1$.
 Assume that there exist matrix units $a_1,\dots, a_m$ such that $p(a_1,\dots,a_m)=\diag\{c_0,\dots,c_{n-1}\}$.
This can be written as a linear combination of the $e_k$.
 Assume also that the image of $p$ is at most $(n^2-n+2)$-dimensional.
By  Remark~\ref{dim-f1},  the image of $f$ constructed in the proof of Theorem~\ref{thmB1}
 is at most $2$-dimensional and thus is a linear space.
 If $p(a_1,\dots,a_m)=h_0e_0+h_1e_1+\dots+h_{n-1}e_{n-1}$ with $h_k\neq 0$, then
  $e_k$ belongs to the linear span  of $\Image f$.
 Hence there are at most two nonzero coefficients, say, $h_k$ and $h_l$ with all of the
 others
 zero.
 We can consider the scalar product
$$ \langle \{ \alpha _1,  \dots, \alpha _n  \} \{ \beta _1,  \dots, \beta _n \}\rangle
= \sum _{i=1}^n \alpha_i \overline{\beta_i}.$$
 We compute $\langle \{c_0,\dots,c_{n-1}\},e_s \}\rangle $ in two ways, first as
 $nq_s$
 and then as
 $$c_0+c_1\varepsilon^{-s}+\dots+c_{n-1}\varepsilon^{-(n-1)s}$$
 since $\overline{\varepsilon^{l}}=\varepsilon^{-l}.$
\end{rem}

\begin{theorem}\label{no-mpc}
 Assume $n\geq 4.$ For $ \Char(K)$ arbitrary,  there are no multilinear power central polynomials.
 Any multilinear polynomial is either PI, or central or its image is at least $(n^2-n+2)$-dimensional.
\end{theorem}
\begin{proof}
 Assume that $p$ is an $n$-central polynomial. Then there exist a set of matrix units $a_i$ such that 
 $p(a_1,\dots,a_m)=\diag\{c_0,\dots,c_{n-1}\}$
 is diagonal but not scalar.  
 Then the image of $f$ as constructed in \eqref{mapping_sets} must be $1$-dimensional, implying 
 $\diag\{c_0,\dots,c_{n-1}\}$ is proportional to $\diag\{c_i,c_{i+1},\dots,c_{n-1},c_0,\dots,c_{i-1}\}$ for any~$i$.
If at least one of the $c_i$ were zero,
 then each $c_i$ would be zero and therefore this matrix is zero (in particular it is scalar), a contradiction. Thus, 
 without loss of generality $c_0\neq c_1$ are nonzero.
 Therefore there also exist a set of matrix units $\tilde a_i$ such that $p(\tilde a_1,\dots,\tilde a_m)=\diag\{c_1,c_0,c_2,\dots,c_{n-1}\}$.
We can construct the mappings $f$ and $\tilde f$ as before, and their images cannot be both $1$-dimensional since otherwise 
  $$\tau=\frac{c_1}{c_0}=\frac{c_2}{c_1}=\frac{c_3}{c_2}=\dots=\frac{c_n}{c_{n-1}}=\frac{c_0}{c_n},$$ 
 and also $\tilde\tau=\frac{c_0}{c_1}=\frac{c_2}{c_0}=\frac{c_3}{c_2}= \cdots.$ Hence   $$\tau^2=\frac{c_2}{c_1}\cdot\frac{c_1}{c_0}=\frac{c_2}{c_0}=\tilde\tau=\frac{c_3}{c_2}=\tau.$$ Thus $\tau \in \{0,1\}$. 
 If $\tau=1$ then
 $p(a_1,\dots,a_m)$ is scalar, a contradiction. If $\tau=0$ then $c_1=0$, a contradiction.
 We conclude that  $\Image f$ is least $2$-dimensional and  $ \Image p$ is at least $(n^2-n+2)$-dimensional. In particular
 $p$ is not  power central.
\end{proof}

Let us  improve the estimates of the dimension of $\Image p$, for $n\ge 5$.

\begin{theorem} \label{harmonic-thm} Suppose the ground field $K$ is as in Remark~\ref{remharmonic}.
 Let $p$ be any multilinear polynomial evaluated on $n\times n$ matrices which is not PI or central.
 Assume that the characteristic of $K$
does not   divide $n$.
 If $n\geq 5$, then the image of $p$ is at least $(n^2-n+3)$-dimensional.
\end{theorem}
\begin{proof}
 The polynomial $p$ is neither PI nor central thus
 there exist matrix units $a_1,\dots, a_m$ such that $p(a_1,\dots,a_m)=\diag\{c_0,\dots,c_{n-1}\}$ is diagonal but not scalar.
 Assume that the image of $p$ is at most $(n^2-n+2)$-dimensional.
 As we showed in Remark \ref{remharmonic}, the matrix $\diag\{c_0,\dots,c_{n-1}\}$
 can be written as $\alpha e_k+\beta e_l$,
which is not scalar. Without loss of generality we may assume that
$c_0\neq c_1$
 (because there exists $r$ such that
 $c_r\neq c_{r+1})$, and we now consider the matrix $$\diag\{c_r,c_{r+1},\dots,c_{n-1}\,c_0,c_1,\dots,c_{r-1}\}$$ instead of our.
with its different coefficients   $\tilde
\alpha=\varepsilon^{rk}\alpha$ and $\tilde \beta
=\varepsilon^{rl}\beta $).

We define the matrices
$q_k:=\diag\{\varepsilon^k,1,\varepsilon^{2k},\varepsilon^{3k},\dots,
 \varepsilon^{(n-1)k}\}.$
 Switching the indices 1 and 2,  we obtain
 matrix units $\tilde a_i$ such that
 $$p(\tilde a_1, \dots,\tilde a_m)=\diag\{c_1,c_0,c_2,c_3,\dots,c_{n-1}\} = \alpha q_k+\beta
 q_l.$$
By Remark~\ref{remharmonic}, $\alpha q_k+\beta  q_l$ also can be written as a linear combination of two elements
  of the base $e_s$
 (say, $\tilde \alpha e_{\tilde k}+\tilde \beta e_{\tilde l}$).
 Note that
 $$\langle q_k,e_s\rangle=\varepsilon^k+\varepsilon^{-s}-1-\varepsilon^{k-s}=
 (\varepsilon^k-1)(1-\varepsilon^{-s})+ \langle  e_k,e_s\rangle .$$

Thus, if $k\neq s$, then $$\langle
q_k,e_s\rangle=\varepsilon^k+\varepsilon^{-s}-1-\varepsilon^{k-s}=
 (\varepsilon^k-1)(1-\varepsilon^{-s})$$ since $\langle
 e_k,e_s\rangle= 0,$
 and if $k=s$ then $$\langle q_k,e_s\rangle=(\varepsilon^k-1)(1-\varepsilon^{-s})+n.$$
 Hence, if $s\notin\{k,l\}$ then $$\langle \alpha q_k+\beta q_l,e_s\rangle=(1-\varepsilon^{-s})(\alpha(\varepsilon^k-1)+\beta (\varepsilon^l-1)).$$
 We denote $\delta =\alpha(\varepsilon^k-1)+\beta (\varepsilon^l-1).$
 Recall  that  $c_1\neq c_0$, and thus $\delta \neq 0$. Therefore either $s=\tilde k$, or else $s=\tilde l$, or
 $\langle \alpha q_k+\beta q_l,e_s\rangle=0$ (and thus
 $s$ is either $k$, or $l$, or $1-\varepsilon^{-s}=0$ (and thus $s=0$) - only five possibilities.
 But for $n\geq 6$ there are at least three nonzero coefficients, a contradiction.

 Thus we may assume that $n=5$.
 We have exactly five possibilities for $s$, which therefore must be distinct.
 Therefore the $k$-th and $l$-th coefficients of $\alpha
  q_k+\beta q_l$ will be zero, i.e.,
 $$(1-\varepsilon^{-k})\delta +5\alpha=(1-\varepsilon^{-l})\delta +5\beta =0,$$
 where $\delta =\alpha(\varepsilon^k-1)+\beta (\varepsilon^l-1).$
 In particular $$\frac{\alpha}{\beta }=\frac{1-\varepsilon^{-k}}{1-\varepsilon^{-l}}.$$
 Now let us take matrix units $a_i'$ such that $p(a_1',\dots,a_m')=\diag\{c_2,c_1,c_0,c_3,c_4\}.$
 Then $\alpha r_k+\beta r_l$ can also be written as a linear combination of two of the $e_s$, where
 $r_k=\diag\{\varepsilon^{2k},\varepsilon^k,1,\varepsilon^{3k},\varepsilon^{4k}\}.$
If $k\neq s$, then $$\langle
r_k,e_s\rangle=\varepsilon^{2k}+\varepsilon^{-2s}-1-\varepsilon^{2k-2s}=
 (\varepsilon^{2k}-1)(1-\varepsilon^{-2s}).$$
 We perform the same calculations as before, and  obtain
 $$\frac{\alpha}{\beta }=\frac{1-\varepsilon^{-2k}}{1-\varepsilon^{-2l}}.$$
 Hence $$\frac{1-\varepsilon^{-k}}{1-\varepsilon^{-l}}=
 \frac{1-\varepsilon^{-2k}}{1-\varepsilon^{-2l}},$$ implying
   $$\frac{1+\varepsilon^{-k}}{1+\varepsilon^{-l}}=1,$$ and hence $k=l$, a contradiction.
  \end{proof}

\begin{theorem}
 Let $p$ be any multilinear polynomial evaluated on $4\times 4$ matrices, which is neither PI nor central. 
 Assume that $\Char K \ne 2$. 
 Then  $\dim \Image p \ge 14,$ equality holding only if the following conditions are satisfied:

 \begin{itemize}
\item For any matrix units $a_i$, if $p(a_1,\dots,a_m)$ is diagonal 
 then it has eigenvalues $(c,c,-c,-c)$ for some $c$.
\item
Any value of $p$ has eigenvalues 
 $(\lambda_1,\lambda_2,-\lambda_1,-\lambda_2).$ 
\end{itemize}
\end{theorem}
\begin{proof} First note that $4^2-4+2 = 14,$ so  $\dim \Image p \ge 14.$
 Assume that $p$ is a multilinear polynomial evaluated on $4\times 4$ matrices with $14$-dimensional image.
 Let $a_1,\dots,a_m$ be any matrix units such that $p(a_1,\dots,a_m)$ is diagonal but not scalar.
 Let $p(a_1,\dots,a_m)=\diag\{c_0,c_1,c_2,c_3\}$ and $c_0\neq c_1$.
 We  use the same notation as in the proof of Theorem \ref{harmonic-thm}.
 Recall that $e_k=\diag\{1,i^k,i^{2k},i^{3k}\}$ and $q_k=\diag\{i^k,1, i^{2k},i^{3k}\}$.
 As     in the proof of Theorem \ref{harmonic-thm},
 $\langle \alpha q_k+\beta q_l , e_s\rangle= \delta (1-i^{-s})$ if $s\notin\{k,l\}$,
 or $\delta (1-i^{-s})+4\alpha $ if $s=k$ and 
 $\delta (1-i^{-s})+4\beta $ if $s=l$.
 Therefore  $k$ and $l$ are nonzero (for otherwise we have two nonzero possibilities for $s\notin \{k,l\}$ 
 and one other nonzero coefficient would be zero: $\langle \alpha q_k+\beta q_l , e_0\rangle=4\alpha$ (if we assume $k=0$ without loss of 
 generality).
 Therefore $p(a_1,\dots,a_m)$ belongs to the linear span $\langle e_1,e_2,e_3\rangle$.
 Hence we have three options: 
 \begin{itemize}
\item $p(a_1,\dots,a_m)=\alpha  e_1+\beta e_2$,
\item $p(a_1,\dots,a_m)=\alpha  e_1+\beta e_3$,
\item $p(a_1,\dots,a_m)=\alpha  e_3+\beta e_2$.
\end{itemize}
We will not treat the last case since its calculations are  as in the first case.
Let us consider the first case $p(a_1,\dots,a_m)=\alpha e_1+\beta e_2$.
Therefore $p(\tilde a_1,\dots,\tilde a_m)=\alpha q_1+\beta q_2$ which can be written explicitly as
$$\left(\frac{1}{2}\alpha-\frac{1+i}{2}\beta\right)e_1+\frac{i-1}{2}\alpha e_2+\left(\frac{i}{2}\alpha+\frac{i-1}{2}\beta\right)e_3,$$
thus $\alpha\in\{0,(1+i)\beta,-(1+i)\beta\}.$
If $\alpha=(1+i)\beta$ then $p(a_1,\dots,a_m)=\beta\diag\{i-2,i+2,-i,-i\}$.
If $\alpha=-(1+i)\beta$ then $p(a_1,\dots,a_m)=\beta\diag\{-i,-i,i+2,i-2\}$.
In both cases $\alpha=\pm (1+i)\beta$, so there are matrix units $\tilde a_i$ such that 
$p(\tilde a_1,\dots,\tilde a_m)=\diag\{-i,i+2,-i,i-2\}$ which can be written explicitly as
$-ie_1-ie_2+ie_3$, and we have three nonzero coefficients. We conclude that the image is at least $15$-dimensional.
If $\alpha=0$, then the value $p(a_1,\dots,a_m)$ has eigenvalues $(c,c,-c,-c)$ as in the conditions of the Theorem.

Assume now $p(a_1,\dots,a_m)=\alpha e_1+\beta e_3$.
Therefore $p(a_1,\dots,a_m)=\diag\{x,y,-x,-y\}.$
Then we consider $p(\tilde a_1,\dots,\tilde a_m)=\diag\{x,-x,y,-y\}$ which can be written explicitely as
$$\frac{(x-y)(1+i)}{4}e_1+\frac{x+y}{2}e_2+\frac{(x-y)(1-i)}{4}e_3,$$
therefore $x=\pm y$ and once again we have a matrix from the conditions of the theorem.

Therefore there is a set of matrix units $a_i$ with $p(a_1,\dots,a_m)=\diag\{c,c,-c,-c\}.$
Now we construct the mapping $f$ and the image will be $2$-dimensional if and only if it is the set 
$\diag\{\lambda_1,\lambda_2,-\lambda_1,-\lambda_2\}.$
Therefore  $\Image p$ contains all the matrices with such eigenvalues, which is a $14$-dimensional variety.
Hence,  if $\dim\Image p = 14,$ then $\Image p$ is exactly this variety.
\end{proof}

\section{Open problems}

The following questions remain from this paper:
\begin{prm}\label{3s}
Does there actually exist a multilinear 3-central polynomial for
$n=3$? We can prove that these do not exist (although homogenous
ones do exist).
\end{prm}

\begin{prm}\label{3s}
Does there exist an $n$-central polynomial (not multilinear) for
$n>3$ prime? This would answer the celebrated cyclicity question
for division algebras of prime degree.
\end{prm}

\end{document}